\documentclass[10pt]{amsart}

\usepackage{amssymb}

\newtheorem{thm}{Theorem}[section]
\newtheorem{prop}[thm]{Proposition}
\newtheorem{lem}[thm]{Lemma}
\newtheorem{cor}[thm]{Corollary}
\theoremstyle{defi}
\newtheorem{defi}[thm]{Definition}

\theoremstyle{rem}
\newtheorem{rem}[thm]{Remark}

\numberwithin{equation}{section}

\newcommand{\R}{\mathbb{R}^n}  

\begin{document}

\title{Smooth solutions to a class of nonlocal fully nonlinear elliptic equations}

\author{Hui Yu}
\address{Department of Mathematics, The University of Texas at Austin}
\email{hyu@math.utexas.edu}

\begin{abstract}
We show that a certain class of fully nonlinear nonlocal equations have smooth solutions as long as the right-hand sides are nice and the boundary data are bounded. To this end we follow the classical strategy. We first show that solutions are $C^{\sigma+\alpha}$ continuous, then develop a bootstrap argument robust enough for operators that may depend on the solution and its $\sigma$-derivatives.

\end{abstract}

\maketitle

\tableofcontents

\section{Introduction}

Since the seminal work of Caffarelli and Silvestre \cite{CS1}, the regularity theory for viscosity solutions to nonlocal fully nonlinear elliptic equations has received much attention. These equations are of the type 

$$\begin{cases} I(u,x)=f(x) &\text{in $\Omega$,} \\
u=g &\text{in $\Omega^c$.}\end{cases}$$

The ellipticity condition is characterised by comparison with two extremal operators 

$$ M^{-}_{\mathcal{L}}(u-v)(x)\le I(u,x)-I(v,x)\le M^{+}_{\mathcal{L}}(u-v)(x).$$ Here 

$$M^-_{\mathcal{L}}(u)(x)=\inf_{K\in\mathcal{L}}\int \delta u(x,y) K(y)dy,$$ where $\mathcal{L}$ is some family of kernels, and $\delta u(x,y)=u(x+y)+u(x-y)-2 u(x)$. A typical example of $\mathcal{L}=\mathcal{L}_0$ is the set of kernels comparable to the fractional Laplacian, namely, 
$$ \frac{\lambda (2-\sigma)}{|y|^{n+\sigma}}\le K(y)\le \frac{\Lambda (2-\sigma)}{|y|^{n+\sigma}}$$ with $0<\lambda<\Lambda<\infty$. The upper extremal operator is defined by changing $\inf$ to $\sup$.
The notion of viscosity solutions is defined by testing the operator on touching smooth functions. See \cite{CS1} for more detail.  

In a series of works \cite{CS1}\cite{CS2}, Caffarelli and Silvestre proved fundamental results about regularity of viscosity solutions to these equations analogous to those in second order equations, for instance, H\"older continuity of solutions, $C^{1,\alpha}$ continuity of solutions for translation invariant operators,  as well as a perturbative type argument for non-translation-invariant operators.  Then in \cite{CS3} the same authors proved the Evans-Krylov theorem for nonlocal equations, hence established the existence of classical solutions for concave operators with nice kernels. 

One interesting feature of these works is that the estimate recovers the second order case as the order $\sigma$ tends to 2, indicating that nonlocal and local equations form a continuum. Also note that these results were later extended to more general types of equations by Kriventsov \cite{Kri} and Serra \cite{Ser}.  T. Jin and J. Xiong also give a different proof in \cite{JX} for the $C^{\sigma+\alpha}$-estimate with a Schauder-type estimate.

In case of second order equations, one can then bootstrap and show that solutions are smooth at least when the operator is smooth. In fact, for solutions to 
$ F(D^2u)=0,$ its derivative satisfies a linear equation with H\"older continuous coefficients $$ F_{ij}(D^2u(x))D_{ij}u_e=0.$$ From here Schauder theory gives the smoothness of the solution. 

For nonlocal equations, however,  this issue is more delicate. The main difficulty is that when coefficients depend on the solution in a global fashion, say, on $\delta u(x,y)$, the very wild boundary behaviour \cite{RosSer}   for nonlocal equations appears and pollutes the estimate. 

To be concrete, consider an equation  of the form $$ I(u,x)=\int  F(\delta u(x,y))\frac{1}{|y|^{n+\sigma}}dy=0,$$ which is a nice example of a nonlocal fully nonlinear elliptic equation  if $F$ is increasing.  Let's assume we have $C^{\sigma+\alpha}$ estimate and focus on higher regularity of the solution. The natural strategy is to differentiate and study the equation for directional derivatives of $u$, say, in the direction of $e$, which solves $$\int  F'(\delta u(x,y))\delta u_e(x,y)\frac{1}{|y|^{n+\sigma}}dy=0.$$ We immediately see that the coefficient $F'(\delta u(x,y))$ depends on global data. In particular, it is at best H\"older continuous \cite{RosSer} no matter how regular the solution is in the interior of the domain of concern. 

This equation is recently studied in \cite{Yu}, where the author showed the existence of  smooth solutions once we pose the problem in the entire $\R$, giving strong evidence that this `dependence on global data' is the only obstruction toward higher regularity for nonlocal equations.

However, in domains other than $\R$, not very many nonlinear nonlocal equations have been shown to admit smooth solutions. Actually up to the knowledge of the author, the only general class of nonlinear nonlocal operators that have smooth solutions are given by Barrios, Figalli and Valdinoci in \cite{BFV}, where the authors studied a nonlocal version of the minimal surface equation. After some analysis this reduced to an operator of the form 
$$ I(u,x)=\int K(x,y) \delta u(x,y)dy. $$ Then they developed a Schauder type theory for such linear equations, and showed the solution to the nonlocal minimal surface equation is smooth once the right-hand side is smooth and the boundary datum is bounded.  One should note that their operator $I$ is linear with the coefficients independent of the solution. This is crucial to their Schauder theory. And the question remains open of whether there is a general class of nonlocal fully nonlinear equations that admit smooth solutions. 

In this paper, we present a general class of nonlocal fully nonlinear elliptic operators that admit smooth solutions. To be precise, our main result is the following
\begin{thm} Let $\rho$ be a nice weight, and $F: S^n\times B_1\to \mathbb{R}$ be a smooth, elliptic operator that is concave in the matrix variable, where $S^n$ is the space of $n\times n$ symmetric matrices. Also assume $F(0,x)=0$ for all $x\in B_1$.  

For $\sigma\ge\sigma_0\in (0,2)$, the Dirichlet problem 
\begin{equation}\begin{cases} F(D^\sigma_\rho u, x)=f(x) &\text{in $B_1$}\\ u=g &\text{in $B_1^c$}\end{cases}\end{equation} admits a unique solution $u$ that is smooth in $B_1$ if $f$ is smooth, and $g\in \mathcal{L}^{\infty}(\R \backslash B_1)$ has some modulus of continuity.
\end{thm} 

\begin{rem}
Given a nice weight $\rho$, the operator $D^\sigma_{\rho}$ is our substitute for the Hessian matrix as in second order equations, its definition given in Section 2, together with certain conditions on possible weights. It is a nonlocal object that carries information about the $\sigma$-order directional derivatives of $u$. 
\end{rem}

\begin{rem}We point out that the uniform continuity assumption on $g$ is only needed for the existence of the solution, and plays no role in regularity estimates. 
In particular, if one knows a priori that a solution exists then the continuity condition on $g$ can be dropped, and still  the solution must be unique and is smooth in the interior of the ball.

The author also wants to point out that in a previous version of the work $g$ needs to be H\"older. Xavier Ros-Oton gave an insightful suggestion on how to reduce the dependence on $g$ to be only on the $\mathcal{L}^{\infty}$ level.
\end{rem}

\begin{rem}With method in this paper, it is not too difficult to address operators involving lower order terms, that is, operators of the form $F(D^\sigma_\rho u, u, x)$ or $F(D^\sigma_\rho u, Du, u, x)$ if $\sigma>1$. However we leave the details to the reader since the proof is already rather technical as it is.
\end{rem}

\begin{rem}
For the definition of a fully nonlinear elliptic operator, see \cite{CC}.
\end{rem}

 Comparing this result to the second order theory, where a concave operator that depends smoothly on second order derivatives admits smooth solutions, our theorem says if a concave operator depends smoothly on the $\sigma$-order derivatives, then it has smooth solutions. 
 
 However, due to the nonlocal nature of the problem, we can allow a much wider range of possible `$\sigma$-order derivatives', instead of a single Hessian. This variety lies in the possible weight $\rho$ which specifies how long range interactions influence the derivative in a bounded domain. In that sense, our operator has two levels of `variable coefficients', the first level is the variable dependence on $D^{\sigma}_\rho u$, which also appears for second order operators as the variable dependence on the Hessian. The second level is how the `Hessian' $D^{\sigma}_\rho u$ depends on $u$ via the weight $\rho$. This second level of `variable coefficients' is a purely nonlocal phenomenon.

To compare with \cite{BFV}, our operators are fully nonlinear.  As a result a more careful localization argument is required.  On the other hand, our result improves  \cite{BFV}, since we do not require the order of the equation $\sigma$ to be above $1$.  This is done by a more careful study of the difference quotients, instead of jumping to first derivatives directly.

Our proof consists of two steps. The first step is to show that the solution is in $C^{\sigma+\alpha}$. This follows from the concavity of the operator by a deep result of Serra \cite{Ser}.  The next step is to bootstrap the solution to $C^{\infty}$. Here the basic idea is that the principle order coefficients are as regular as $\Delta^{\sigma/2} u$, for which we can use the estimate from the previous step, and this regularity for the coefficients is enough for a bootstrap argument.

This paper is organized as follows: in Section 2 we present some preliminary results on nonlocal operators as well as some definitions that are useful for the rest of the paper; in Section 3 we present the $C^{\sigma+\alpha}$ estimate; in Section 4 we complete the proof of the theorem by giving a bootstrap argument; and then in Section 5 we present some applications of our general theory.

\section{Preliminaries}
We start by giving two definitions related to  the operator $D^\sigma_{\rho}$: 
\begin{defi}For $\sigma\in (0,2)$, we define the $|\cdot|_{\sigma, \Omega}$ seminorm by 
$$|u|_{\sigma,\Omega}=\sup_{x\in\Omega}\int |\delta u(x,y)|\frac{2-\sigma}{|y|^{n+\sigma}}dy.$$

The scaled version of this seminorm is 
$$|u|_{\sigma,\Omega}^{*}=\sup_{x\in\Omega}d(x)^{\sigma}\int |\delta u(x,y)|\frac{2-\sigma}{|y|^{n+\sigma}}dy,$$ where $d$ is the distance function to $\partial\Omega$.

\end{defi} 
This is in some sense a $C^\sigma$ seminorm suitable for the study of $\sigma$ order nonlocal operators. Like the standard H\"older seminorms it obeys an interpolation inequality
\begin{prop}
Given $\epsilon>0$ there is a constant $C(\epsilon)<\infty$ such that 
$$|u|_{\sigma,B_1}^{*}\le \epsilon\|u\|_{C^{\sigma+\alpha}(B_1)}^*+C(\epsilon)\|u\|_{\mathcal{L}^{\infty}(\R)},$$here $\|u\|_{C^{\sigma+\alpha}(B_1)}^*$ is the scaled H\"older norm. \end{prop} 
See \cite{GT} for a definition of scaled H\"older norms.
\begin{proof}
Suppose, on the contrary, that there is no such $C(\epsilon)$ for some $\epsilon$. Then we find a sequence $u_n$ such that 
$$|u_n|_{\sigma,B_1}^*\ge \epsilon \|u_n\|_{C^{\sigma+\alpha}(B_1)}^*+n\|u_n\|_{\mathcal{L}^{\infty}(\R)}.$$ After a normalization we assume $ \|u_n\|_{C^{\sigma+\alpha}(B_1)}^*=1$ for all $n$, and in particular $|u_n|_{\sigma,B_1}^*\le C$ . This gives a subsequence $u_k$ converging locally uniformly in $B_1$ to some $u_\infty$, and we necessarily have $\|u_\infty\|_{\mathcal{L}^{\infty}(\R)}\to 0$. 
Hence $u_k=0$. This leads to a contradiction since  $|u_\infty|_{\sigma,B_1}^*\ge \epsilon$ by our assumption.
\end{proof}

For functions with finite $|\cdot|_{\sigma, \{x\}}$ seminorm, one can define a $\sigma$ order replacement for the Hessian matrix. 

\begin{defi}
Given $\sigma\in (0,2)$, $\rho:B_1\times \R\to [\lambda_{\rho},\Lambda_{\rho}]$, a smooth function bounded away from $0$ and $\infty$, we define $D^\sigma_{\rho} u(x)$ to be  the matrix with the $(i,j)$-entry
$$ D^\sigma_{\rho,ij}u(x)=\int \delta u(x,y)\frac{(2-\sigma)\langle y,e_i\rangle\langle y,e_j\rangle \rho(x,y)}{|y|^{n+\sigma+2}}dy,$$ where $\{e_j\}$ is the standard basis of $\R$, here $u\in\mathcal{L}^\infty(\R)$ and $|u|_{\sigma,\{x\}}<\infty$.

\end{defi} 

\begin{rem}
We simplify $D^{\sigma}_{\rho}$ to $D^{\sigma}$ whenever there is no ambiguity.
\end{rem} 

\begin{rem}
As mentioned in Introduction, this is our substitute for the Hessian matrix. The very general weight $\rho$ introduces much richer structure than possible in the Hessian. Basically the only requirement on $\rho$ is that it stays away from $0$ and $\infty$, this is to balance the influence from different directions and positions, and hence the constants $\lambda_{\rho}$, and $\Lambda_{\rho}$ can also be viewed as elliptic constants. Together with the elliptic condition on the operator $F$, they give the regularization of the equation. 

All constants depending only on the dimension $n$, the elliptic constants of $F$, $\lambda_{\rho}$ and $\Lambda_{\rho}$ are universal constants.\end{rem} 

\begin{rem}
The coefficient $2-\sigma$ is for the estimate to be uniform for big $\sigma$.
\end{rem}

The map $D^\sigma$ interacts with regularity classes in the same way as $\Delta^{\sigma/2}$: 
\begin{prop}
If $u\in\mathcal{L}^{\infty}(\R)\cap C^{\sigma+\alpha}(B_1)$, then $D^\sigma u\in C^{\alpha}(B_{1/4})$ with estimates.
\end{prop} 

\begin{proof}
Let $\eta$ be a smooth cut-off that is $1$ in $B_{3/4}$ and vanishes outside $B_1$. Take $u^1=u\eta$ and $u^2=u(1-\eta)$. 

Since $\|u^1\|_{C^{\sigma+\alpha}(\R)}\le C\|u\|_{C^{\sigma+\alpha}(B_1)}$, one has 
\begin{align*}|D^{\sigma}_{ij}u^1(x)-D^{\sigma}_{ij}u^1(x')|&\le|\int(\delta u^1(x,y)-\delta u^1(x',y))\frac{(2-\sigma)\langle y,e_i\rangle\langle y,e_j\rangle \rho(x,y)}{|y|^{n+\sigma+2}}dy|\\&+|\int\delta u^1(x',y)\frac{(2-\sigma)\langle y,e_i\rangle\langle y,e_j\rangle (\rho(x,y)-\rho(x',y))}{|y|^{n+\sigma+2}}dy|\\&=I_1+I_2.\end{align*}
Since $\rho$ is bounded by standard estimate for the fractional Laplacian one has $$I_1\le C\|u\|_{C^{\sigma+\alpha}(B_1)}|x-x'|^{\alpha}.$$ 
 As for $I_2$, we use the smoothness of $\rho$ in $x$ $$|\rho(x,y)-\rho(x',y)|\le [\rho(\cdot,y)]_{C^{0,1}_x}|x-x'|^{\alpha}\le C|x-x'|^{\alpha}$$ and $$|\int\delta u^1(x',y)\frac{(2-\sigma)\langle y,e_i\rangle\langle y,e_j\rangle}{|y|^{n+\sigma+2}}dy|\le C\|u\|_{C^{\sigma+\alpha}(B_1)}$$ to get $$I_2\le C\|u\|_{C^{\sigma+\alpha}(B_1)}|x-x'|^{\alpha}.$$ 


On the other hand, for $x,x'=x+h\in B_{1/4}$, $\delta u^2(x,y)=\delta u^2(x+h,y)=0$ unless $|y|>1/2$. Also, $u^2(x)=u^2(x+h)=0$. Hence 
\begin{align*} |D^\sigma_{ij}u^2(x')-D^\sigma_{ij}u^2(x)| &\le |\int_{|y|>1/2}(u^2(x+h+y)-u^2(x+y))\frac{(2-\sigma)\langle y,e_i\rangle\langle y,e_j\rangle}{|y|^{n+\sigma+2}}dy|\\ &+|\int_{|y|>1/2} u^2(x+h-y)-u^2(x-y))\frac{(2-\sigma)\langle y,e_i\rangle\langle y,e_j\rangle}{|y|^{n+\sigma+2}}dy|\\ &\le \int_{|y|>1/4}|u^2(x+y)||\frac{(2-\sigma)\langle y-h,e_i\rangle\langle y-h,e_j\rangle}{|y-h|^{n+\sigma+2}}-\frac{(2-\sigma)\langle y,e_i\rangle\langle y,e_j\rangle}{|y|^{n+\sigma+2}}|dy\\&+\int_{|y|>1/4}|u^2(x-y)||\frac{(2-\sigma)\langle y+h,e_i\rangle\langle y+h,e_j\rangle}{|y+h|^{n+\sigma+2}}-\frac{(2-\sigma)\langle y,e_i\rangle\langle y,e_j\rangle}{|y|^{n+\sigma+2}}|dy\\& \le C\|u\|_{\mathcal{L}^\infty(\R)}|h|\int_{|y|>1/4}|\nabla_y\frac{(2-\sigma)\langle y,e_i\rangle\langle y,e_j\rangle}{|y|^{n+\sigma+2}}|dy\\ &\le C\|u\|_{\mathcal{L}^\infty(\R)}|h|.\end{align*}

Then we see $\| D^{\sigma}_{ij}u \|_{C^{\alpha}(B_{1/4})}$ is bounded by a combination of $\|u\|_{\mathcal{L}^{\infty}}$ and $\|u\|_{C^{\sigma+\alpha}(B_1)}$.
\end{proof}

Now we show that our fully nonlinear operator is elliptic in the sense of \cite{CS1}:

\begin{prop}
$F(\cdot,x)$ is uniformly elliptic with respect to $\mathcal{L}_0$.
\end{prop}  

\begin{proof}Take $u,v\in \mathcal{L}^{\infty}(\R)\cap C^2(x)$, then both $D^{\sigma}u$ and $D^{\sigma}v$ are defined at $x$. 

\begin{align*}
F(D^\sigma u,x)-F(D^{\sigma}v, x) &\le\mathcal{M}^{+}_2(D^{\sigma}u(x)-D^{\sigma}v(x))\\ &\le\sup_{\lambda\le A\le\Lambda}A_{ij}(D^{\sigma}_{ij}u(x)-D^{\sigma}_{ij}v(x))\\ &=\sup_{\lambda\le A\le\Lambda}A_{ij}\int (\delta u(x,y)-\delta v(x,y))\frac{(2-\sigma)\langle y,e_i\rangle\langle y,e_j\rangle\rho(x,y)}{|y|^{n+\sigma+2}}dy\\ &=\sup_{\lambda\le A\le\Lambda}\int (\delta u(x,y)-\delta v(x,y))\frac{(2-\sigma)\langle A y,y\rangle\rho(x,y)}{|y|^{n+\sigma+2}}dy\\ &\le (2-\sigma)\int(\Lambda(\delta u(x,y)-\delta v(x,y))^{+}-\lambda(\delta u(x,y)-\delta v(x,y))^{-})\frac{\rho(x,y)}{|y|^{n+\sigma}}dy\\ &\le C(\lambda_{\rho},\Lambda_{\rho})\mathcal{M}^+_{\mathcal{L}_0(\sigma)}(u-v)(x).
\end{align*}

Here $\mathcal{M}^{+}_2$ denote the maximal Pucci operator. 

Similarly one shows $F(D^\sigma u,x)-F(D^{\sigma}v, x)\ge C(\lambda_{\rho},\Lambda_{\rho})\mathcal{M}^-_{\mathcal{L}_0(\sigma)}(u-v)(x)$.
\end{proof}

We now deal with existence and uniqueness of the solution. This important point is pointed out to the author by Dennis Kriventsov.
\begin{prop}
If $f\in C(B_1)$ and $g\in\mathcal{L}^{\infty}(\R)$ has some modulus of continuity, then
(1.1) admits a unique viscosity solution $u\in C(B_1)\cap \mathcal{L}^{\infty}(\R)$.
\end{prop} 
\begin{proof} See Section 5 of \cite{Kri} for a proof of existence via a regularization and approximation procedure. The uniqueness follows because our solution is smooth, and we have comparison principle for smooth solutions \cite{CS1}. \end{proof} 

These results indicate that  $D^\sigma_{\rho}$ is a nice replacement for the Hessian matrix in nonlocal equations. However to bootstrap to smoothness, we need to impose the following condition on the possible weights we consider.

\begin{defi}
A smooth bounded weight $\rho: B_1\times\R\to [\lambda_{\rho},\Lambda_{\rho}]$ is nice if it satisfies the following decay properties in $y$: $$|D^{\alpha}_y\rho(x,y)||y|^{|\alpha|}\le M_{|\alpha|},$$ where $ M_{|\alpha|}$ is a constant depending only on $|\alpha|$, the length of the multi-index $\alpha$.
\end{defi}

\begin{rem}
This decay is needed for condition (3.2) in the next section. The author thanks the careful referee for pointing this out.
\end{rem}

\section{$C^{\sigma+\alpha}$ estimate}
We now exploit the concavity of the operator $F$ to show that our solution is $C^{\sigma+\alpha}$ in the interior of the domain. This is done in two steps. In the first step, we show that solutions with finite $|\cdot|_{\sigma,B_1}$ seminorm are actually $C^{\sigma+\alpha}$. Then this improvement of regularity is combined with the interpolation to give the estimate we need.

For the first step, we use a deep result by Serra \cite{Ser} concerning variable coefficient concave equations:

\begin{thm}
For $\sigma\in(0,2)$ let $I$ be an operator of the form 
$$I(u,x)=\inf_a(\int \delta u(x,y)K_a(x,y)dy+c_a(x))$$ where each $c_a$ is H\"older, each $K_a(x,\cdot)\in\mathcal{L}_0(\sigma)$ and enjoys a H\"older continuity in $x$ 
\begin{equation}\int_{B_{2r}\backslash B_{r}}|K_a(x,y)-K_a(x',y)|dy\le A_0|x-x'|^{\alpha}\frac{2-\sigma}{r^\sigma}\end{equation} as well as a decay in $y$ 
\begin{equation}[K_a(x,\cdot)]_{C^{\alpha}(B_\rho^c)}\le\Lambda (2-\sigma)\rho^{-n-\sigma-\alpha}\end{equation} for some $A_0$  and $\alpha\in (0,1)$.  Then there is a universal $\overline{\alpha}$ such that if $\alpha<\overline{\alpha}$ the viscosity solution to 
$$\begin{cases} I(u,x)=0 &\text{in $B_1$,} \\
u=g &\text{in $B_1^c$}\end{cases}$$ satisfies the following estimate
$$\|u\|_{C^{\sigma+\alpha}(B_{1/2})}\le C(\sup_a \|c_a\|_{C^\alpha(B_1)}+\|g\|_{\mathcal{L}^{\infty}(\R)} )$$where the constant $C$ depends on universal constants, $\alpha$, as well as $A_0$.
\end{thm} 

Now we prove the following 

\begin{lem} There is a universal $\overline{\alpha}>0$ such that for any $\alpha<\overline{\alpha}$ 
the solution $u$ to (1.1) with $$|u|_{\sigma,B_1}\le 1$$ satisfies 
$$\|u\|_{C^{\sigma+\alpha}(B_{1/2})}\le C(1+\|f\|_{C^\alpha(B_1)}+\|g\|_{\mathcal{L}^{\infty}(\R)}+\|\partial_x F(\cdot,\cdot)\|_{\infty}+\sup_{N\in S^n}[D_MF(N,\cdot)]_{C^{0,1}_x} ),$$where the constant $C$ depends on universal constants and $\sup_{N\in S^n}[D_MF(N,\cdot)]_{Lip}$.
\end{lem} 
Here $D_MF(N,x)$ is the differential of $F$ with respect to the matrix variable at $(N,x)$.

\begin{proof}
Due to its smoothness and concavity our operator can be written in the following form 
$$F(M,x)=\inf_{N\in S^n}\{D_MF(N,x)M-D_MF(N,x)N+F(N,x)\}.$$  As a result, if we define the kernels 
$$K_N(x,y)=\frac{(2-\sigma)\langle D_MF(N,x)y,y\rangle\rho(x,y)}{|y|^{n+\sigma+2}}$$ then the solution satisfies the equation 
$$\inf_{N\in S^n}(\int \delta u(x,y)K_N(x,y)dy+c_N(x))=0$$ where $c_N(x)=-D_MF(N,x)N+F(N,x)-f(x)$.

Since $|u|_{\sigma,B_1}\le 1$, and in particular $ \|D^{\sigma}u(x)\|\le \Lambda_{\rho}$ in $B_1$, the equation can be further reduced to 
$$\inf_{\|N\|\le  \Lambda_{\rho}}(\int \delta u(x,y)K_N(x,y)dy+c_N(x))=0.$$

To check condition (3.1) we note that 
\begin{align*}|K_N(x,y)-K_N(x',y)|&=\frac{2-\sigma}{|y|^{n+\sigma+2}}|\rho(x,y)\langle D_MF(N,x)y,y\rangle-\rho(x',y)\langle D_MF(N,x')y,y\rangle|\\&\le\frac{2-\sigma}{|y|^{n+\sigma+2}}|\rho(x,y)\langle(D_MF(N,x)-D_MF(N,x'))y,y\rangle|\\&+\frac{2-\sigma}{|y|^{n+\sigma+2}}||(\rho(x,y)-\rho(x',y))\langle D_MF(N,x')y,y\rangle|\\&\le\frac{2-\sigma}{|y|^{n+\sigma+2}}\Lambda_\rho[D_MF(N,\cdot)]_{C^{0,1}_x}|x-x'||y|^2\\&+\frac{2-\sigma}{|y|^{n+\sigma+2}}|[\rho(\cdot,y)]_{C^{0,1}_x}|x-x'|\Lambda|y|^2\\&\le C\frac{2-\sigma}{|y|^{n+\sigma}}|([\rho(\cdot,y)]_{C^{0,1}_x}+[D_MF(N,\cdot)]_{C^{0,1}_x})|x-x'|. \end{align*} Note that we used the ellipticity of $\rho$ and $F$. 

Consequently we can take $A_0$ to be $C(\sup_{N\in S^n}[D_AF(N,\cdot)]_{Lip}+\sup_{y\in\R}[\rho(\cdot,y)]_{C^{0,1}_x})$ and (3.1) is true for any $\alpha\le 1.$

The other condition (3.2) follows from 
\begin{align*}|\nabla_y K_N(x,y)|&=|\frac{(2-\sigma)}{|y|^{n+\sigma+2}}\rho(x,y)\nabla_y(\langle D_MF(N,x)y,y\rangle)\\&+\frac{(2-\sigma)}{|y|^{n+\sigma+2}}\langle D_MF(N,x)y,y\rangle\nabla_y \rho(x,y)\\&+\rho(x,y)\langle D_MF(N,x)y,y\rangle\nabla_{y}(\frac{(2-\sigma)}{|y|^{n+\sigma+2}})|\\&=|\frac{(2-\sigma)}{|y|^{n+\sigma+2}}\rho(x,y)2\langle D_MF(N,x)y\\&+\frac{(2-\sigma)}{|y|^{n+\sigma+2}}\langle D_MF(N,x)y,y\rangle\nabla_y \rho(x,y)\\&+\rho(x,y)\langle D_MF(N,x)y,y\rangle\frac{(2-\sigma)(n+\sigma+2)y}{|y|^{n+\sigma+4}}|\\&\le|\frac{(2-\sigma)2\Lambda\Lambda_{\rho}}{|y|^{n+\sigma+1}}|+|\frac{(2-\sigma)\Lambda M_1}{|y|^{n+\sigma+1}}|+|\frac{(2-\sigma)(n+\sigma+2)\Lambda\Lambda_{\rho}}{|y|^{n+\sigma+1}}|.\end{align*} Here $M_1$ is the constant in Definition 2.10.

Hence to apply Theorem 3.1 we only need to estimate the H\"older norm of $c_N$. To this end, we have

\begin{align*}|c_N(x)|&\le \|-D_MF(N,x)\|\cdot\|N\|+|F(N,x)|+|f(x)|\\ &\le \Lambda\|N\|+\Lambda\|N\|+\|f\|_{\mathcal{L}^{\infty}(B_1)}\\&\le 2\Lambda\Lambda_{\rho}+\|f\|_{\mathcal{L}^{\infty}(B_1)}.\end{align*} Here we used the fact that $F(0,x)=0$ and that $\|N\|\le \Lambda_\rho$.

Also,
\begin{align*}|c_N(x)-c_N(x')| &\le |D_MF(N,x)-D_MF(N,x')|\|N\|+|F(N,x)-F(N,x')|+|f(x)-f(x')|\\ &\le\Lambda_\rho \sup_{N\in S^n}[D_MF(N,\cdot)]_{C^{0,1}_x}|x-x'|+\sup_{N\in S^n}\|\partial_x F(N,\cdot)\||x-x'|+[f]_{C^{\alpha}(B_1)}|x-x'|^{\alpha}.\end{align*}

In conclusion Theorem 3.1 gives

$$\|u\|_{C^{\sigma+\alpha}(B_{1/2})}\le C(1+\|f\|_{C^\alpha(B_1)}+\|\partial_x F(\cdot,\cdot)\|_{\infty}+\sup_{N\in S^n}[D_MF(N,\cdot)]_{C^{0,1}_x}+\|g\|_{\mathcal{L}^{\infty}(\R)} ),$$ where the constant $C$ depends on universal constants, $\alpha<\overline{\alpha}$, $\sup_{N\in S^n}[D_MF(N,\cdot)]_{C^{0,1}_x}$ and $\sup_{y\in\R}[\rho(\cdot,y)]_{C^{0,1}_x}$.

\end{proof}

A scaling and interpolation argument gives the $C^{\sigma+\alpha}$ estimate:

\begin{prop}
Let $u$ and $F$ be as in Theorem 1.1. Then there is a universal $\overline{\alpha}$ such that $$\|u\|_{C^{\sigma+\alpha}(B_{1/2})}\le C(\|f\|_{C^\alpha(B_1)}+\|\partial_x F(\cdot,\cdot)\|_{\infty}+\|g\|_{\mathcal{L}^{\infty}(\R)} )$$ where the constant $C$ depends on universal constants , $\alpha<\overline{\alpha}$, $\sup_{N\in S^n}[D_MF(N,\cdot)]_{C^{0,1}_x}$ and $\sup_{y\in\R}[\rho(\cdot,y)]_{C^{0,1}_x}$.
\end{prop} 

\begin{proof}
After a regularization procedure, see Section 2 in \cite{CS2}, one can always assume  $|u|_{\sigma,B_1}$ is finite. Define a renormalization of $u$ $$w=u/|u|_{\sigma,B_1}.$$ Then $w$ solves the same type of equation with right-hand side $\frac{1}{|u|_{\sigma,B_1}}f$, boundary datum $\frac{1}{|u|_{\sigma,B_1}}g$ and a new operator $G:S^n\times B_1\to\mathbb{R}$ defined as $$G(A,x)=\frac{1}{|u|_{\sigma,B_1}}F(|u|_{\sigma,B_1}A,x).$$ Note that this is still a smooth, concave, elliptic operator with the same ellipticity constants. Moreover, $$\partial_x G(N,\cdot)=\frac{1}{|u|_{\sigma,B_1}}\partial_x F(|u|_{\sigma,B_1}N,\cdot)$$ and $$ D_M G(N,x)=D_M F(|u|_{\sigma,B_1}N,x).$$ Thus the lemma gives

\begin{align*}\|w\|_{C^{\sigma+\alpha}(B_{1/2})}&\le C(1+\frac{1}{|u|_{\sigma,B_1}}\|f\|_{C^\alpha(B_1)}+\|\partial_x G(\cdot,\cdot)\|_{\infty}\\&+\sup_{N\in S^n}[D_MG(N,\cdot)]_{C^{0,1}_x}+\frac{1}{|u|_{\sigma,B_1}}\|g\|_{\mathcal{L}^{\infty}(\R)})\\&\le C(1+\frac{1}{|u|_{\sigma,B_1}}\|f\|_{C^\alpha(B_1)}+\frac{1}{|u|_{\sigma,B_1}}\|\partial_x F(\cdot,\cdot)\|_{\infty}\\&+\sup_{N\in S^n}[D_MF(N,\cdot)]_{C^{0,1}_x}+\frac{1}{|u|_{\sigma,B_1}}\|g\|_{\mathcal{L}^{\infty}(\R)} ).\end{align*}

Consequently,
$$\|u\|_{C^{\sigma+\alpha}(B_{1/2})}\le C(|u|_{\sigma,B_1}+\|f\|_{C^\alpha(B_1)}+\|\partial_x F(\cdot,\cdot)\|_{\infty}+|u|_{\sigma,B_1}\sup_{N\in S^n}[D_MF(\cdot,\cdot)]_{Lip}+\|g\|_{\mathcal{L}^{\infty}(\R)} ).$$

Now we can choose $\epsilon<\frac{1}{2}C(1+\sup_{N\in S^n}[D_MF(N,\cdot)]_{C^{0,1}_x})$ in the interpolation inequality to get the desired estimate.
\end{proof}

\section{Bootstrap to $C^{\infty}$}

According to the previous section, our solution is classical with pointwisely well-defined $D^{\sigma}u(x)$ at least in the interior of the domain. In this section we use a bootstrap strategy to show that this solution is actually smooth in the interior. We  could assume the operator $F$, the weight $\rho$ and the right-hand side $f$ to be in a certain class, say, $C^m$,  and show that the solution belongs to $C^{m+\sigma+\alpha}$. However this would require a careful tracking of all the constants involved and is rather off the point of this paper. Hence we only present the proof for the smoothness of the solution when the operator, the weight and the right-hand side are  all assumed to be smooth. A careful reader can adapt the strategy for less regular operators and/or right-hand sides when necessary.

As mentioned in the Introduction, our method is closely modelled after \cite{BFV}, with a more careful localization procedure. The basic idea is: Suppose $u$ solves an equation in some original  domain $\Omega_0$. Then by the structure of the equation we show that $u$ is in some regularity class $C_1$ in $\Omega_1\subset \Omega_0$. To improve this class we use a cut-off $u^1$ that only sees the part of $u$ inside $\Omega_2''\subset \Omega_1$, hence in particular $u^1$ is in $C_1$, and if the coefficient depends on the solution in a nice way then our $u^1$ would solve a better equation inside some $\Omega_2'\subset \Omega_2''$. This would give a better regularity estimate on $u^1$ inside $\Omega_2\subset \Omega_2'$ in a better class $C_2$. And since our solution $u$ agrees with $u^1$ inside $\Omega_2$, this shows that $u$ itself belongs to $C_2$ in $\Omega_2$. Suppose we can have a strict improvement ($C_2>C_1$) each time without losing too much space ($\Omega_2$ is not too small relative to $\Omega_1$), then this can be iterated until we hit the class of smooth functions.

\subsection{First bootstrap step}By rescaling we have the following estimate 
$$\|u\|_{C^{\sigma+\alpha_1}(B_{R_1})}\le C(R_1,\alpha_1)( \|g\|_{\mathcal{L}^{\infty}(\R)}+\|f\|_{C^{\alpha_1}(B_1)}+\|\partial_x F(\cdot,\cdot)\|_{\infty})$$ where $C$ also depends on  $\sup_{N\in S^n}[D_MF(N,\cdot)]_{C^{0,1}_x}$. Here we fix an $R_1>1/2$ and some $\alpha_1<\overline{\alpha}$. For the localization procudure we further choose some $0<R_2<R_2'<R_2''<R_1$, and let $\eta$ be a cut-off function that is $1$ on $B_{R''_2}$ and vanishes outside $B_{R_1}$. Note that here $B_{R_1}$ is the domain where we have nice estimates, $B_{R_{2}''}$ is the region where the cut-off agrees with our solution, $B_{R_2'}$ is where we have a new equation for the cut-off, and $B_{R_2}$ is domain where we want an improvement of regularity on $u$.

Fix $e\in\mathbb{S}^{n-1}$, for $x\in B_{R'_2}$, $|h|<\frac{1}{4}(R_2''-R_2')$ and $\beta=\min\{\sigma+\alpha_1,1\}$, we define the difference quotient $$w_h(x)=\frac{u(x+he)-u(x)}{|h|^\beta},$$ and further $$w^1_h(x)=\eta(x)w_h(x)$$ and  $$w^2_h(x)=(1-\eta(x))w_h(x).$$ Note that we are taking the $\beta$ order difference quotient because we did not assume $\sigma>\sigma_0>1$. Otherwise we could just take the usual difference quotient.

The $C^{\sigma+\alpha_1}$ estimate gives 
\begin{lem}
$$\|w^1_h\|_{\mathcal{L}^\infty(\R)}\le C(R_1,\alpha_1)( \|g\|_{\mathcal{L}^{\infty}(\R)}+\|f\|_{C^{\alpha_1}(B_1)}+\|\partial_x F(\cdot,\cdot)\|_{\infty}).$$
\end{lem} 

Moreover it solves an elliptic equation:

By taking the difference of $F(D^{\sigma}u,x+he)=f(x+he)$ and  $F(D^{\sigma}u,x)=f(x)$ one sees that for $x\in B_{R_2'}$, 
$$A^h_{ij}(x)(D^{\sigma}_{ij}u(x+he)-D^{\sigma}_{ij}u(x))=f(x+he)-f(x)+k_h(x)h,$$where 
$$A^h_{ij}(x)=\int_0^1\frac{\partial F}{\partial M_{ij}}(tD^\sigma u (x+he)+(1-t)D^\sigma u (x), x+the)dt,$$and $$k_h(x)=\int_0^1\partial_xF(tD^\sigma u(x+he)+(1-t)D^\sigma u(x), x+the)dt\cdot e.$$

Hence if we define the matrix $A^h(x)=(A^h_{ij}(x))$ then one has
\begin{align*}\int(\delta u(x+he,y)-\delta u(x,y))&\frac{(2-\sigma)\langle A^h(x)y,y\rangle \rho(x+he,y)}{|y|^{n+\sigma+2}}dy\\&+\int\delta u(x,y)(\rho(x+he,y)-\rho(x,y))\frac{(2-\sigma)\langle A^h(x)y,y\rangle}{|y|^{n+\sigma+2}}dy\\&=f(x+he)-f(x)+k_h(x) \end{align*}

As a result the difference quotient solves, in $B_{R_2'}$, the following equation 
\begin{align*}\int \delta w_h(x,y)&\frac{(2-\sigma)\langle A^h(x)y,y\rangle\rho(x+he,y)}{|y|^{n+\sigma+2}}dy=\frac{f(x+he)-f(x)}{|h|^\beta}+k_h(x)|h|^{1-\beta}\\&-\int\delta u(x,y)\frac{\rho(x+he,y)-\rho(x,y)}{|h|^{\beta}}\frac{(2-\sigma)\langle A^h(x)y,y\rangle}{|y|^{n+\sigma+2}}dy\\&= \frac{f(x+he)-f(x)}{|h|^\beta}+k_h(x)|h|^{1-\beta}-I^1_h(x).\end{align*} Thus the cut-off satisfies
\begin{align*}
\int \delta w^1_h(x,y)\frac{(2-\sigma)\langle A^h(x)y,y\rangle\rho(x+he,y)}{|y|^{n+\sigma+2}}dy&=\frac{f(x+he)-f(x)}{|h|^\beta}+k_h(x)|h|^{1-\beta}-I^1_h(x)\\&-\int \delta w^2_h(x,y)\frac{(2-\sigma)\langle A^h(x)y,y\rangle\rho(x+he,y)}{|y|^{n+\sigma+2}}dy\\&=\frac{f(x+he)-f(x)}{|h|^\beta}+k_h(x)|h|^{1-\beta}-I^1_h(x)-I^2_h(x).
\end{align*} 

To conclude we have\begin{lem}
$w_h^1$ is a bounded continuous function solving  \begin{equation}
\int \delta w^1_h(x,y)\frac{(2-\sigma)\langle A^h(x)y,y\rangle\rho(x+he,y)}{|y|^{n+\sigma+2}}dy=\frac{f(x+he)-f(x)}{|h|^\beta}+k_h(x)|h|^{1-\beta}-I^1_h(x)-I^2_h(x)\end{equation} in $B_{R_2'}$, where $$I^1_h(x)= \int\delta u(x,y)\frac{\rho(x+he,y)-\rho(x,y)}{|h|^{\beta}}\frac{(2-\sigma)\langle A^h(x)y,y\rangle}{|y|^{n+\sigma+2}}dy$$ and $$I_h^2(x)=\int \delta w^2_h(x,y)\frac{(2-\sigma)\langle A^h(x)y,y\rangle\rho(x+he,y)}{|y|^{n+\sigma+2}}dy.$$
\end{lem} 
We want to apply Theorem 3.1 to this equation. For the various estimate the following is useful.  It says the regularity of the coefficients improves because of  previous estimate on the solution.

\begin{lem}
$A^h\in C^{\alpha_1}(B_{R_2'})$ with estimate independent of $h$ and $\beta$.
\end{lem}  

\begin{proof}
Obviously $\|A^h\|_{\infty}\le \|F\|_{C^1}$. For the H\"older semi-norm
\begin{align*}
|A^h_{ij}(x)-A^h_{ij}(x')|&\le \int_0^1|\frac{\partial F}{\partial M_{ij}}(tD^\sigma u (x+he)+(1-t)D^\sigma u (x), x+the)\\&-\frac{\partial F}{\partial M_{ij}}(tD^\sigma u (x'+he)+(1-t)D^\sigma u (x'), x'+the)|dt\\&\le \|F\|_{C^2}(|D^{\sigma}u(x+he)-D^{\sigma}u(x'+he)|+|D^{\sigma}u(x)-D^{\sigma}u(x')|+|x-x'|)\\&\le C\|F\|_{C^2}(\|u\|_{C^{\sigma+\alpha_1}(B_{R_1})}|x-x'|^{\alpha_1}+|x-x'|)\end{align*} for some $C$ independent of $h$ and $\beta$. Note that we used Proposition 2.6 for the last inequality.
\end{proof}

Now we check conditions (3.1) and (3.2).
For condition (3.1), note that 
\begin{align*}|\frac{\langle A^h(x)y,y\rangle\rho(x+he,y)}{|y|^{n+\sigma+2}}&-\frac{\langle A^h(x')y,y\rangle\rho(x'+he,y)}{|y|^{n+\sigma+2}}|\le\frac{1}{|y|^{n+\sigma+2}}|\rho(x+he,y)\langle (A^h(x)-A^h(x'))y,y\rangle|\\&+\frac{1}{|y|^{n+\sigma+2}}|\langle A^h(x')y,y\rangle(\rho(x+he,y)-\rho(x'+he,y))|\\ &\le \frac{\Lambda_\rho}{|y|^{n+\sigma}}|A^h(x)-A^h(x')|\\&+\frac{|A^h(x)|}{|y|^{n+\sigma}}|(\rho(x+he,y)-\rho(x'+he,y)|\\&\le\frac{\Lambda_\rho\|F\|_{C^2}}{|y|^{n+\sigma}}(|x-x'|+2[D^\sigma u]_{C^{\alpha_1}(B_{R_2'})}|x-x'|^{\alpha_1})+\frac{\Lambda}{|y|^{n+\sigma}}[\rho(\cdot,y)]_{C^{0,1}_x}|x-x'|.
\end{align*}  This gives condition (3.1) with H\"older exponent $\alpha_1$ with estimate independent of $h$ and $\beta$. 

Condition (3.2) again follows from the decay condition on $\rho$.


Now we estimate the H\"older norm of the right-hand side. The $\alpha_1$ H\"older norm of $\frac{f(\cdot+he)-f(\cdot)}{|h|^\beta}$ is clearly bounded by the $C^{1,1}$ norm of $f$. The estimate for the other terms takes more effort.

\begin{lem}
$k_h\in C^{\alpha_1}(B_{R_2'})$ with estimate independent of $h$ and $\beta$.
\end{lem} 

\begin{proof}
Obviously $\|h_k\|_{\infty}\le\|\partial_x F\|_\infty.$  For the H\"older semi-norm one uses the smoothness of $F$ and the $C^{\sigma+\alpha_1}$-estimates together with Proposition 2.6.
\begin{align*}
|k_h(x)-k_h(x')|&=|\int_0^1 (\partial_x F(tD^\sigma u(x+he)+(1-t)D^\sigma u(x),x+the)-\\&\partial_x F(tD^\sigma u(x'+he)+(1-t)D^\sigma u(x'),x'+the)dt|\\&\le \|F\|_{C^2}(|D^\sigma u(x+he)-D^\sigma u(x'+he)|+|D^\sigma u(x)-D^\sigma u(x')|+|x-x'|)\\&\le C\|F\|_{C^2}([u]_{C^{\sigma+\alpha_1}(R_1)}|x-x'|^{\alpha_1}+|x-x'|).
\end{align*} Here the constants also depends on $R_1, R_2'$ as well as regularity of $\rho$, but nevertheless independent of $h$ and $\beta$.
\end{proof}

\begin{lem}
$I^1_h\in C^{\alpha_2}(B_{R_2'})$ for some $\alpha_2>0$ with estimate independent of $h$ and $\beta$.
\end{lem} 

\begin{proof}
For $x\in B_{R_2'}$,
\begin{align*}
|I^1_h(x)|&\le \int_{|y|\le 1/8(R_1-R_2')}|\frac{\rho(x+he,y)-\rho(x,y)}{|h|^{\beta}}||\delta u(x,y)|\frac{(2-\sigma)\langle A^h(x)y,y\rangle}{|y|^{n+\sigma+2}}|dy\\&+ \int_{|y|> 1/8(R_1-R_2')}|\frac{\rho(x+he,y)-\rho(x,y)}{|h|^{\beta}}||\delta u(x,y)\frac{(2-\sigma)\langle A^h(x)y,y\rangle}{|y|^{n+\sigma+2}}|dy\\&\le C\sup_y[\rho(\cdot,y)]_{C^{0,1}_x}(\|u\|_{C^{\sigma+\alpha_1}(B_{R_1})}+\|u\|_{\infty}).\end{align*} Here $C$ is independent of $h$ and $\beta$.
Also note that $$I^1_h(x)=A^h_{ij}(x)\int\delta u(x,y)\frac{\rho(x+he,y)-\rho(x,y)}{|h|^{\beta}} \frac{(2-\sigma)\langle y,e_i\rangle\langle y,e_j\rangle}{|y|^{n+\sigma+2}}dy=A^h_{ij}(x)\cdot J(x),$$ where we already has a $C^{\alpha_1}$-estimate on $A^h_{ij}$. As a result, it suffices to prove a H\"older estimate on $J$.

To see this note
\begin{align*}
|J(x)-J(x')|&\le |\int(\delta u(x,y)-\delta u(x',y))\frac{\rho(x+he,y)-\rho(x,y)}{|h|^{\beta}} \frac{(2-\sigma)\langle y,e_i\rangle\langle y,e_j\rangle}{|y|^{n+\sigma+2}}dy|\\&+|\int\delta u(x',y)(\frac{\rho(x+he,y)-\rho(x,y)}{|h|^{\beta}}-\frac{\rho(x'+he,y)-\rho(x',y)}{|h|^{\beta}}) \frac{(2-\sigma)\langle y,e_i\rangle\langle y,e_j\rangle}{|y|^{n+\sigma+2}}dy|.
\end{align*} The first term is bounded by $\sup_y[\rho(\cdot,y)]_{C^{0,1}_x}\|u\|_{C^{\sigma+\alpha_1}(B_{R_1})}|x-x'|^{\alpha_1}$, while the second is bounded by $\sup_y\|\rho(\cdot,y)\|_{C^2_x}(\|u\|_{C^{\sigma+\alpha_1}(B_{R_1})}+\|u\|_{\infty})|x-x'|,$independent of $h$ and $\beta$.
Thus one can take $\alpha_2<\alpha_1^2$.
\end{proof}

Finally we estimate the H\"older norm of $I_h^2$. Here the key idea is to exploit the smoothness of the kernel in $y$, since we are away from the singularity due to the cut-off.

\begin{lem}
$I^2_h\in C^{\alpha_1}(B_{R_2'})$ with estimate independent of $h$ and $\beta$.
\end{lem}

\begin{proof} 
Note that
\begin{align*}\delta w^2_h(x,y) =& \frac{1}{|h|^\beta}[(1-\eta(x+y))u(x+y+he)-(1-\eta(x+y))u(x+y)\\ & +(1-\eta(x-y))u(x-y+he)-(1-\eta(x-y))u(x-y)\\ &+2(1-\eta(x))u(x)-2(1-\eta(x))u(x+he)]\\=& \frac{1}{|h|^\beta}[(1-\eta(x+y))u(x+y+he)-(1-\eta(x+y))u(x+y)\\ & +(1-\eta(x-y))u(x-y+he)-(1-\eta(x-y))u(x-y)].\end{align*} Therefore the expansion of $I_h$ consists of two similar terms, one of which is 
\begin{align}& \frac{2-\sigma}{|h|^\beta}\int_{|y|>R_2''-R_2'} [(1-\eta(x+y))u(x+y+he)-(1-\eta(x+y))u(x+y)]\frac{\langle A^h(x)y,y\rangle\rho(x+he,y)}{|y|^{n+\sigma+2}}dy =\\ &\frac{2-\sigma}{|h|^\beta}\int_{|y|>3/4(R_2''-R_2')} u(x+y)[(1-\eta(x+y-he))\frac{\langle A^h(x)(y-he),y-he\rangle\rho(x+he,y-he)}{|y-he|^{n+\sigma+2}}\\ &-(1-\eta(x+y))\frac{\langle A^h(x)y,y\rangle\rho(x+he,y)}{|y|^{n+\sigma+2}}]dy.\end{align}
 
  Since $\partial_y [(1-\eta(x+y))\frac{\langle A^h(x)y,y\rangle\rho(x+he,y)}{|y|^{n+\sigma+2}}]$ is bounded by $\frac{C}{|y|^{n+\sigma+1}}$, 
  $$\frac{1}{|h|^{\beta}} [(1-\eta(x+y-he))\frac{\langle A^h(x)(y-he),y-he\rangle}{|y-he|^{n+\sigma+2}}-(1-\eta(x+y))\frac{\langle A^h(x)y,y\rangle}{|y|^{n+\sigma+2}}]\le \frac{C|h|^{1-\beta}}{|y|^{n+\sigma+1}},$$
  thus $|I^2_h(x)|\le C \|u\|_{\mathcal{L}^{\infty}(\R)}$ independent of $h$ and $\beta$.  
  
  To estimate its H\"older seminorm, we note that  
  
  \begin{align*}& \frac{2-\sigma}{|h|^\beta}\int_{|y|>3/4(R_2''-R_2')} u(x+y)[(1-\eta(x+y-he))\frac{\langle A^h(x)(y-he),y-he\rangle\rho(x+he,y-he)}{|y-he|^{n+\sigma+2}}\\ &-(1-\eta(x+y))\frac{\langle A^h(x)y,y\rangle\rho(x+he,y)}{|y|^{n+\sigma+2}}]dy\\&= \frac{2-\sigma}{|h|^\beta}\int_{y\in B^c_{3/4(R_2''-R_2')}(x)} u(y)[(1-\eta(y-he))\frac{\langle A^h(x)(y-x-he),y-x-he\rangle\rho(x+he,y-x-he)}{|y-x-he|^{n+\sigma+2}}\\ &-(1-\eta(y))\frac{\langle A^h(x)(y-x),(y-x)\rangle\rho(x+he,y-x)}{|y-x|^{n+\sigma+2}}]dy.\end{align*}
  
  And if $x,x'\in B_{R_2'}$, 
  
  \begin{align*}&|[(1-\eta(y-he))\frac{\langle A^h(x)(y-x-he),y-x-he\rangle\rho(x+he,y-x-he)}{|y-x-he|^{n+\sigma+2}}\\&-(1-\eta(y))\frac{\langle A^h(x)(y-x),(y-x)\rangle\rho(x+he,y-x)}{|y-x|^{n+\sigma+2}}]\\-& [(1-\eta(y-he))\frac{\langle A^h(x')(y-x'-he),y-x'-he\rangle\rho(x'+he,y-x'-he)}{|y-x'-he|^{n+\sigma+2}}\\&-(1-\eta(y))\frac{\langle A^h(x')(y-x'),(y-x')\rangle\rho(x'+he,y-x')}{|y-x'|^{n+\sigma+2}}]|\\ \le& \frac{C_{\eta,\rho}}{|y|^{n+\sigma}}|h|(|A^h(x)-A^h(x')|)+\frac{C}{|y|^{n+\sigma}}|h|\|\rho\|_{C^1}|x-x'|+C_{\eta,\rho}|h|(|\frac{1}{|y-x|^{n+\sigma}}-\frac{1}{|y-x'|^{n+\sigma}}|)\\ \le& \frac{C_{\eta,\rho}}{|y|^{n+\sigma}}|h| \|A^h\|_{C^{\alpha_1}(B_{R_2'})}|x-x'|^{\alpha_1}+\frac{C}{|y|^{n+\sigma}}|h|\|\rho\|_{C^1}|x-x'|+\frac{C_{\eta,\rho}}{|y|^{n+\sigma+1}}|h||x-x'|. \end{align*}
  
Note that if $R_2'<\frac{7}{15}R_2''$, then $ B^c_{3/4(R_2''-R_2')}(x)\cup  B^c_{3/4(R_2''-R_2')}(x')\supset B^c_{1/8(R_2''-R_2')} $  
hence we can avoid the singularity of the integral at $x$ and $x'$ at the same time, and hence 
\begin{align*}|I^2_h(x)-I^2_{h}(x')|&\le (2-\sigma)|h|^{1-\beta}( \|A^h\|_{C^{\alpha_1}(B_{R_2'})}+\|\rho\|_{C^1}+1)|x-x'|^{\alpha_1}\\&\cdot\int_{|y|>1/8(R_2''-R_2')}\frac{C_{\eta,\rho}}{|y|^{n+\sigma}}+\frac{C_{\eta,\rho}}{|y|^{n+\sigma+1}}dy\\ &\le C|x-x'|^{\alpha_1},\end{align*} with $C$ independent of $h$ and $\beta$.
  \end{proof}
 
 Now all the conditions in Theorem 3.1 are satisfied, hence we have the following estimate 
 
\begin{equation*}\|w^1_h\|_{C^{\sigma+\alpha_2}(B_{R_2})}\le C_{\rho} (\|f\|_{C^{1,1}(B_1)}+\|D^2F\|_{\infty}+\|\partial_x F\|_{\infty}+\|u\|_{\mathcal{L}^{\infty}(\R)}+\|u\|_{C^{\sigma+\alpha_1}(B_{R_1})} +1).
\end{equation*}

This being true for all $h$ and $\beta=\min\{\sigma+\alpha_1,1\}$ we see we can iterate to get the following, by taking a smaller $R_2$ and $\alpha_3$ even smaller if necessary, 

\begin{align*}
\|u\|_{C^{1+\sigma+\alpha_3}(B_{R_2})}\le &C (\|f\|_{C^{1,1}(B_1)}+\|D^2F\|_{\infty}(1+\|D^\sigma u\|_{C^{\alpha_1}(B_{R_2'})})\\&+\|u\|_{\mathcal{L}^{\infty}(\R)}+\|D^\sigma u\|_{C^{\alpha_1}(B_{R_2'})}+\|u\|_{C^{\sigma+\alpha_1}(B_{R_1})} +1)
\end{align*} Using the estimate from the previous section, this reduces to 

\begin{prop} There is $\alpha_3\in (0,1)$ such that for $R_2$ small,   
\begin{equation*} 
\|u\|_{C^{1+\sigma+\alpha_3}(B_{R_2})}\le C( \|g\|_{\mathcal{L}^{\infty}(\R)}+\|f\|_{C^{1,1}(B_1)}+\|\partial_x F(\cdot,\cdot)\|_{\infty}+1)
\end{equation*} where the constant $C$ depends on universal constants, $\alpha_3$, $R_2$, $\|F\|_{C^2}$ and regularity of $\rho$.
\end{prop}

Note that then by a standard covering argument we have 
\begin{cor}There is $\alpha_3\in (0,1)$ such that for any $R_2<1$, 
\begin{equation} 
\|u\|_{C^{1+\sigma+\alpha_3}(B_{R_2})}\le C( \|g\|_{\mathcal{L}^{\infty}(\R)}+\|f\|_{C^{1,1}(B_1)}+\|\partial_x F(\cdot,\cdot)\|_{\infty}+1)
\end{equation} where the constant $C$ depends on universal constants, $\alpha_3$, $R_2$, $\|F\|_{C^2}$ and regularity of $\rho$.
\end{cor}

\subsection{Further regularity}Take now $0<R_3<R_3'<R_3''<R_2$, then in $B_{R_3'}$ the estimates in the previous subsection gives enough regularity to take $\beta=1$ and $h\to 0$ in the left-hand side of equation (4.1), which converges to 
$$\int \delta u^1_e(x,y)\frac{(2-\sigma)\langle A(x)y,y\rangle\rho(x,y)}{|y|^{n+\sigma+2}}dy.$$ Here $u^1_e$ is a cut-off in $B_{R_2}$ of the directional derivative of $u$ in the direction $e$. And 
$$A_{ij}(x)=\frac{\partial F}{\partial M_{ij}}(D^\sigma u (x), x).$$
The first term on the right-hand side converges nicely to $f_e(x)$. The second term converges to $k(x)=\partial_x F(D^\sigma u(x),x)\cdot e.$ The third converges to $$I^1(x)=\int\delta u(x,y)\frac{(2-\sigma)\langle A(x)y,y\rangle \partial_{x,e}\rho(x,y)}{|y|^{n+\sigma+2}}dy.$$
Using (4.2) we see that the third term on the right-hand side converges to 
\begin{align*}I^2(x) &=I^{2,1}(x)+I^{2,2}(x)\\ &=(2-\sigma)\int_{|y|>3/4(R_2''-R_2')} u(x+y)\partial_{y,e}[(1-\eta(x+y))\frac{\langle A(x)y,y\rangle\rho(x,y)}{|y|^{n+\sigma+2}}]dy\\&+(2-\sigma)\int_{|y|>3/4(R_2''-R_2')}u(x-y)\partial_{y,e}[(1-\eta(x-y))\frac{\langle A(x)y,y\rangle\rho(x,y)}{|y|^{n+\sigma+2}}]dy.\end{align*} As a result, 

\begin{lem}$u_e^1$ is a bounded function that solves, in $B_{R_3'}$,   the following equation \begin{equation}
\int \delta u^1_e(x,y)\frac{(2-\sigma)\langle A(x)y,y\rangle}{|y|^{n+\sigma+2}}dy=f_e(x)+k(x)-I^1(x)-I^{2,1}(x)-I^{2,2}(x).
\end{equation}\end{lem}
We point out that both the coefficient and the right-hand side enjoy nice regularity in $B_{R_3'}$. For instance 
\begin{lem}
$$\|A\|_{C^{1+\alpha_3}(B_{R_3'})}\le C(\|u\|_{C^{1+\alpha_3+\sigma}(B_{R_2})}+\|u\|_{\infty}),$$
$$\|f_e\|_{C^{1+\alpha_3}(B_{R_3'})}\le \|f\|_{C^{2,1}(B_1)},$$
$$\|k\|_{C^{1+\alpha_3}(B_{R_3'})}\le C\|F\|_{C^2}(\|u\|_{C^{1+\alpha_3+\sigma}(B_{R_2})}+1).$$\end{lem} 
\begin{proof}
Use the smoothness of $F$, $f$ and the estimate from Section 4.1.\end{proof}

The only term that requires extra care is $I^2$, since the terms $u(x+y)$ and $u(x-y)$ see the global behaviour of the solution and hence does not have regularity higher than H\"older. But again we exploit the fact that our kernel is very smooth in $y$ to give the following

\begin{lem} $I^2\in C^{1+\alpha_4}$ for some $\alpha_4>0$ with estimate.\end{lem}

\begin{proof}
Fix a direction $\tilde{e}$, and let $|h|<\frac{1}{4}(R_2''-R_2')$
\begin{align*}
&\frac{I^{2,1}(x+h\tilde{e})-I^{2,1}(x)}{|h|}\\&=\frac{2-\sigma}{|h|}\int_{|y|>3/4(R_2''-R_2')} u(x+h\tilde{e}+y)\partial_{y,e}[(1-\eta(x+h\tilde{e}+y))\frac{\langle A(x+h\tilde{e})y,y\rangle\rho(x+h\tilde{e},y)}{|y|^{n+\sigma+2}}]\\&-u(x+y)\partial_{y,e}[(1-\eta(x+y))\frac{\langle A(x)y,y\rangle\rho(x,y)}{|y|^{n+\sigma+2}}]dy\\&=\frac{2-\sigma}{|h|}\int_{|y|>1/2(R_2''-R_2')} u(x+y)\partial_{y,e}[(1-\eta(x+y))\frac{\langle A(x+h\tilde{e})(y-h\tilde{e}),y-h\tilde{e}\rangle\rho(x+h\tilde{e},y-h\tilde{e})}{|y-h\tilde{e}|^{n+\sigma+2}}]\\&-u(x+y)\partial_{y,e}[(1-\eta(x+y))\frac{\langle A(x)y,y\rangle\rho(x,y)}{|y|^{n+\sigma+2}}]dy\\&=\frac{2-\sigma}{|h|}\int_{|y|>1/2(R_2''-R_2')} u(x+y)\partial_{y,e}[(1-\eta(x+y))\frac{\langle A(x+h\tilde{e})(y-h\tilde{e}),y-h\tilde{e}\rangle\rho(x+h\tilde{e},y-h\tilde{e})}{|y-h\tilde{e}|^{n+\sigma+2}}\\&-(1-\eta(x+y))\frac{\langle A(x)y,y\rangle\rho(x,y)}{|y|^{n+\sigma+2}}]dy.
\end{align*}

Now the product rule gives 
\begin{align*}
\partial_{y,e}&[(1-\eta(x+y))(\frac{\langle A(x+h\tilde{e})(y-h\tilde{e}),y-h\tilde{e}\rangle\rho(x+h\tilde{e},y-h\tilde{e})}{|y-h\tilde{e}|^{n+\sigma+2}}-\frac{\langle A(x)y,y\rangle\rho(x,y)}{|y|^{n+\sigma+2}})]\\&=\partial_{y,e}(1-\eta(x+y))(\frac{\langle A(x+h\tilde{e})(y-h\tilde{e}),y-h\tilde{e}\rangle\rho(x+h\tilde{e},y-h\tilde{e})}{|y-h\tilde{e}|^{n+\sigma+2}}-\frac{\langle A(x)y,y\rangle\rho(x,y)}{|y|^{n+\sigma+2}})\\&+(1-\eta(x+y))\partial_{y,e}(\frac{\langle A(x+h\tilde{e})(y-h\tilde{e}),y-h\tilde{e}\rangle\rho(x+h\tilde{e},y-h\tilde{e})}{|y-h\tilde{e}|^{n+\sigma+2}}-\frac{\langle A(x)y,y\rangle\rho(x,y)}{|y|^{n+\sigma+2}})\\&=\partial_{y,e}(1-\eta(x+y))(\frac{\langle A(x+h\tilde{e})(y-h\tilde{e}),y-h\tilde{e}\rangle\rho(x+h\tilde{e},y-h\tilde{e})}{|y-h\tilde{e}|^{n+\sigma+2}}-\frac{\langle A(x)y,y\rangle\rho(x,y)}{|y|^{n+\sigma+2}})\\&+(1-\eta(x+y))\partial_{y,e}(\frac{\langle A(x+h\tilde{e})(y-h\tilde{e}),y-h\tilde{e}\rangle\rho(x+h\tilde{e},y-h\tilde{e})}{|y-h\tilde{e}|^{n+\sigma+2}})\\&-(1-\eta(x+y))\partial_{y,e}(\frac{\langle A(x)y,y\rangle\rho(x,y)}{|y|^{n+\sigma+2}}).\end{align*}

Dividing this by $|h|$ and letting $|h|\to 0$, it follows by another application of the product rule that
\begin{align*}
\frac{1}{|h|}\partial_{y,e}&[(1-\eta(x+y))(\frac{\langle A(x+h\tilde{e})(y-h\tilde{e}),y-h\tilde{e}\rangle\rho(x+h\tilde{e},y-h\tilde{e})}{|y-h\tilde{e}|^{n+\sigma+2}}-\frac{\langle A(x)y,y\rangle\rho(x,y)}{|y|^{n+\sigma+2}})]\\&\to\partial_{y,e}(1-\eta(x+y))\frac{\langle\partial_{x,\tilde{e}}(A(x)\rho(x,y))y,y\rangle}{|y|^{n+\sigma+2}}+\partial_{y,e}(1-\eta(x+y))\partial_{y,\tilde{e}}(\frac{\langle A(x)y,y\rangle\rho(x,y)}{|y|^{n+\sigma+2}})\\&+(1-\eta(x+y))\partial_{y,e}(\frac{\langle\partial_{x,\tilde{e}}(A(x)\rho(x,y))y,y\rangle}{|y|^{n+\sigma+2}})+(1-\eta(x+y))\partial_{y,\tilde{e}}\partial_{y,e}(\frac{\langle A(x)y,y\rangle\rho(x,y)}{|y|^{n+\sigma+2}}).
\end{align*} 

Like in the proof for Lemma 4.3, we can use the $C^{1+\alpha_3}$ estimate on $A$, the smoothness of $\rho$, and the smoothness of the kernel in $y$, to get  
$\|I^{2,1}\|_{C^{1+\alpha_4}(B_{R_3'})}\le C$. 

Similar argument applies to $I^{2,2}$.\end{proof}

All this regularity allows us to again apply the strategy in Section 4.1 to our new equation (4.6), and to show that $u\in C^{2+\sigma+\alpha_4}(B_{R_3})$. Iterating this localization-bootstrap argument, we see $u$ is smooth in the interior of $B_1$, and we have proved the theorem
\begin{thm}
$u$ is smooth in the interior of $B_1$.
\end{thm} 

\begin{proof}
Directional derivatives of the equation are equations of the same type to which the theory above can be applied to iteratively.
\end{proof} 

\section{Applications}
In this section we give some applications of this general theory.

\subsection{Linear operators on $D^{\sigma}_{\rho}$}
\begin{thm}
Let $A:B_1\to S^n$ be smooth and $0<\lambda\le A(\cdot)\le\Lambda<\infty$. Then the problem 
$$\begin{cases}\int\delta u(x,y)\frac{\langle A(x)y,y\rangle\rho(x,y)}{|y|^{n+\sigma+2}}=f(x) &\text{in $B_1$}\\ u=g &\text{outside $B_1$}\end{cases}$$ has a unique smooth solution if $\rho$ is as in Definition 2.10, $f$ is smooth, $g$ is bounded and enjoys some uniform continuity. 
\end{thm} 

\begin{rem}
This is a direct application of our theory when $F$ is a linear operator. Note the uniform continuity assumption on $g$ is only for the existence, while the smoothness of the solution only requires the boundedness of $g$.
\end{rem}

\subsection{Functions of the eigenvalues of $D^{\sigma}_{\rho}u$} \begin{thm}
Let $f:\R\to \mathbb{R}$ be a smooth concave function satisfying for all $j$
$$\lambda<\frac{\partial f}{\partial \lambda_j}<\Lambda.$$

Then the equation 
\begin{equation*}\begin{cases}F(D^{\sigma}u)=\phi(x) &\text{in $B_1$}\\ u=\beta &\text{outside}\end{cases},
\end{equation*} where $F(M)=f(\lambda_1,\dots,\lambda_n)$ with $\{\lambda_j\}$ being the eigenvalues of $M$, has a unique solution that is smooth in $B_1$ if $\phi$ is smooth and $\beta$ bounded and uniform continuous.
\end{thm} 

\begin{rem}This is a direct application of our main theorem. See \cite{CNS} for the second order theory.  \end{rem}

\subsection{Improvement of regularity}
As mentioned before, our strategy also applies to less regular operators and right-hand side. The careful reader could verify the following \begin{thm} Let $F:S^n\times B_1\to \mathbb{R}$ be concave in the matrix variable, and elliptic. Also assume it is in $C^k$ with estimates. Then the equation
\begin{equation*}\begin{cases}F(D^{\sigma}u,x)=f(x) &\text{in $B_1$}\\ u=g &\text{outside}\end{cases}\end{equation*} has a unique $C^{k-1+\sigma+\alpha}$ solution for some $\alpha>0$ if $f$ is $C^k$ and $g$ is bounded.
\end{thm} 

This can be also applied to quasilinear elliptic equations if one has some a priori estimate on the solution:

\begin{thm}
Let $F:S^n\times\mathbb{R}\times B_1\to \mathbb{R}$ be concave in the matrix variable, and elliptic. Also assume it is in $C^k$ with estimates. Assume a solution to 
\begin{equation*}\begin{cases}F(D^{\sigma}u,u,x)=f(x) &\text{in $B_1$}\\ u=g &\text{outside}\end{cases}\end{equation*}  is $C^k$ in $B_1$, then $u$ is actually $C^{k-1+\sigma+\alpha}$ solution for some $\alpha>0$ if $f$ is $C^k$ and $g$ is bounded.

\end{thm}
\begin{proof}We use the regularity of $u$ to reduce the operator to $\widetilde{F}(M,x)=F(M,u(x),x)$. Then apply the previous theorem.\end{proof} Note that this applies to operators of the form 
$$\int\delta u(x,y)\frac{a(u,x)}{|y|^{n+\sigma}}dy$$ once we have some a priori estimate on $u$, and some regularity of $\lambda\le a\le\Lambda$.

\section*{Acknowledgements}The author would like to thank his PhD advisor, Luis Caffarelli, for many valuable conversations regarding this project. He is grateful to his colleagues and friends, especially Dennis Kriventsov, Luis Duque, Xavier Ros-Oton and Yunan Yang, for all the discussions and encouragement. Dennis Kriventsov and Xavier Ros-Oton also very carefully read a previous version of this work and gave many insightful comments. In particular, as already mentioned, Ros-Oton pointed to the author that regularity of $g$ is not necessary for the main result. Kriventsov pointed to the author the significance of a existence and uniqueness result.  The author would also like to thank the referee of this paper for pointing out several places for clarification.


\end{document}